\documentclass[12pt,reqno]{amsart}

\usepackage{color}
\numberwithin{equation}{section}
\newcommand{\I}{\mathrm{i}}
\newcommand{\D}{\mathrm{d}}

\newcommand{\wh}{\widehat}

\newcommand{\lb}{\left(}

\newcommand{\vp}{\varphi}

\newcommand{\rb}{\right)}
\newcommand{\PD}{\partial}

\newcommand{\wt}{\widetilde}
\newcommand{\Tl}{\tilde}

\newcommand{\Rb}{\mathbb{R}}
\newcommand{\Sb}{\mathbb{S}}

\newcommand{\Beq}{\begin{equation}}
	\newcommand{\Eeq}{\end{equation}}
\newcommand{\beq}{\begin{equation*}}
	\newcommand{\eeq}{\end{equation*}}
\newcommand{\bal}{\begin{align}}
	\newcommand{\eal}{\end{align}}

\newcommand{\n}{\nabla}

\newcommand{\A}{\alpha}
\newcommand{\B}{\beta}
\newcommand{\bp}{\begin{prob}}
	\newcommand{\ep}{\end{prob}}
\newcommand{\bpr}{\begin{proof}}
	\newcommand{\epr}{\end{proof}}

\newcommand{\bel}[1]{\begin{equation}\label{#1}}
	\newcommand{\ee}{\end{equation}}

\newtheorem{theorem}{Theorem}[section]

\newtheorem{lemma}[theorem]{Lemma}

\theoremstyle{definition}

\usepackage{verbatim,pdfsync}
\title[Light ray transform of symmetric 2-tensor fields]{A uniqueness result for light ray transform on symmetric 2-tensor fields}

\author[Krishnan, Senapati and Vashisth]{Venkateswaran P Krishnan$^{\dagger}$, Soumen Senapati$^{\ddagger}$ and Manmohan Vashisth$^{*}$}
\address{$^{\dagger}$ TIFR Centre for Applicable Mathematics, Bangalore 560065, India. 
	\newline\indent E-mail:{\tt \ vkrishnan@tifrbng.res.in}}
\address{$^{\ddagger}$ TIFR Centre for Applicable Mathematics, Bangalore 560065, India. 
	\newline\indent E-mail:{\tt \ soumen@tifrbng.res.in}}
\address{$^{*}$ Beijing Computational Science Research Center, Beijing 100193, China.
	\newline
	\indent E-mail:{\tt\  mvashisth@csrc.ac.cn}}

\begin{document}
	
	\begin{abstract}
		We study light ray transform of symmetric 2-tensor fields defined
		on a bounded time-space domain in $\Rb^{1+n}$  for $n\geq 3$. We prove a uniqueness result for such light ray transforms. More precisely, we characterize the kernel of the light ray transform vanishing near a fixed direction at each point in the time-space domain. 
	\end{abstract}
	\maketitle
\vspace{2mm}

	
	\section{Introduction and statement of the main results}
	Let $S^{2}\Rb^{1+n}$ be the complex vector space of symmetric tensor fields of rank $2$ in $\Rb^{1+n}$. Let $\Omega$ be a bounded domain in $\Rb^{1+n}$ with $C^{\infty}$ boundary and  $C^{\infty}(\overline{\Omega}; S^{2} \Rb^{1+n})$ be the space of $S^{2}\Rb^{1+n}$-valued $C^{\infty}$ smooth symmetric 2-tensor fields on $\overline{\Omega}$.  We represent points in $\Omega$ by $(t,x)$.  
	Any $F\in C^{\infty}(\overline{\Omega}; S^{2} \Rb^{1+n})$ will be denoted by
	\[
	F(t,x)= \lb F_{ij}(t, x)\rb \mbox{ where } 0\leq i,j\leq n \mbox{ with } F_{ij}(t, x)=F_{ji}(t,x) \mbox{ and } F_{ij}(t, x)\in C^{\infty}(\overline{\Omega}).
	\]
	Note that we have used the $0$-index to denote the time component of a symmetric $2$-tensor field. Also note that a function $f\in C^{\infty}(\overline{\Omega})$ if it has a smooth extension to a slightly larger open set containing $\overline{\Omega}$.
	
	The light ray transform of $F\in C^{\infty}(\overline{\Omega}; S^{2}\Rb^{1+n})$ is defined as follows. Consider a point $(t,x)\in \overline{\Omega}$  and  fix a direction $\theta\in \Sb^{n-1}$. The light ray transform $L$ of $F$ is the usual ray transform of $F$ through the point $(t,x)$ in the direction $\wt{\theta}=(1,\theta)$. That is,  
	\begin{align}\label{definition of L}
		LF(t,x,\Tl{\theta})=\int\limits_{\mathbb{R}}\sum\limits_{i,j=0}^{n}\tilde{\theta}^{i}\tilde{\theta}^{j}F_{ij}(t+s,x+s\theta)\, \D s.
	\end{align}

	We have assumed the Einstein summation convention and from now on, with repeating indices, this will be assumed, with the index varying from $0$ to $n$. We also note that extending $F$ to be $0$ outside $\Omega$, the definition of the light ray transform $L$ can be extended to points $(t,x)\in \Rb^{1+n}$ and any $\Tl{\theta}$ as defined above.  
	This will be assumed without comment from now on.
	
	In this work, we address the question of characterizing the tensor fields $F\in C^{\infty}(\overline{\Omega}; S^{2}\Rb^{1+n})$ such that $LF(t,x,\Tl{\theta})=0$ for all $(t,x)\in\Rb^{1+n}$ and all $\Tl{\theta}=(1,\theta)$ with $\theta\in\mathbb{S}^{n-1}$ near some fixed $\theta_{0}\in\mathbb{S}^{n-1}$.  
	
	Light ray transforms in Euclidean and manifold settings have been studied in several recent works; see \cite{feizmohammadi2019light,lassas2019light,LOSU, Siamak_Support_theorem_2017,Stefanov_Support_Theorem_Lorentzian_Manifold_2017,YW}. Most of these works analyze light ray transform from the view point of microlocal analysis. Light ray transforms arise in the study of inverse problems for hyperbolic PDEs with time-dependent coefficients as well; see references \cite{IBA,Eskin_IHP_time-dependent_2007,AJYL,Ali_Kian,Kian_Damping_partial_data_2016,Kian_Unique_determination_potential_partial_data_2017,Krishnan_2019,Ramm_Rakesh_Property_C_1991,Ramm_Sjostrand_IP_wave_equation_potential_1991,Salazar_time-dependent_first_order_perturbation_2013,
		Stefanov1987,Stefanov_Inverse_scattering_potential_time_dependent_1989, Stefanov-Yang}.

	To the best of the authors' knowledge, an exact  description of the kernel of the light ray transform on symmetric 2-tensor fields has not been precisely studied, and this is the main goal of the paper. In this work, we use Fourier transform techniques to prove the uniqueness result in the Minkowski setting. We should mention that the recent paper \cite{feizmohammadi2019light} also deals with uniqueness result for light ray transforms on symmetric tensor fields and a uniqueness result similar to ours is proven in the more general setting of Lorentzian manifolds. Our work was done independently, and the techniques employed here are different from theirs. Specifically, their uniqueness result for the light ray transform on certain Lorentzian manifolds relies on the uniqueness result for the corresponding geodesic ray transform on the base space; see \cite[Theorem 2]{feizmohammadi2019light}. However, we work directly with the light ray transform, albeit in the Minkowski setting. Another distinction from the work of \cite{feizmohammadi2019light} is that our uniqueness result only assumes knowledge of the light ray transform in the neighborhood of a fixed direction. In other words, ours is a partial data result. The uniqueness result of \cite{feizmohammadi2019light}, requires knowledge of the full light ray transform, even in the setting of Minkowski space.

	We now state the main results.
ults as well as in proofs below we use the following notation. By $\delta$, we mean the Euclidean divergence and $\mbox{trace}$ will refer to the Euclidean trace.

	In other words, for a symmetric $2$-tensor field $F=(F_{ij})_{0\leq i,j\leq n}$: 
	\begin{align}
	& \label{divergence defn}\lb \delta F\rb_{i} = \frac{\PD  F_{j0}}{\PD t}+\sum\limits_{j=1}^{n}\frac{\PD F_{ij}}{\PD x_{i}} = \PD_{i} F_{ij}. \\
	& \mbox{trace}(F)=\sum\limits_{i=0}^{n} F_{ii}.
	\end{align}
	
	In the last equality in \eqref{divergence defn}, we emphasize that the Einstein summation convention is assumed with the index varying between $0$ and $n$ and the $\frac{\PD}{\PD t}$ derivative is abbreviated as $\PD_{0}$ and the space derivatives $\frac{\PD}{\PD x_{i}}$ are denoted by $\PD_{i}$ for $1\leq i\leq n$.

	\begin{theorem}\label{Main Theorem for LRT of 2-tensor fields in R3}
		Let $F\in C^{\infty}(\overline{\Omega}; S^2 \Rb^{1+n})$ be such that $\delta F=0$ and $\mathrm{trace}(F)=0$.
		If for a fixed $\theta_{0}\in \Sb^{n-1}$, 
		\begin{align*}
			LF(t,x,\Tl{\theta})=0 \mbox{ for all } (t,x)\in\mathbb{R}^{1+n}\ \mbox{ and }\ \theta \mbox{ near } \theta_{0},
		\end{align*}
		then $F=0$.
	\end{theorem}
	\begin{theorem}\label{DT}
		Let $F\in C^{\infty}(\overline{\Omega}; S^2 \Rb^{1+n})$. Then there exists an $\wt{F}\in C^{\infty}(\overline{\Omega}; S^2 \Rb^{1+n})$ satisfying $\delta(\wt{F})=\mbox{trace}(\wt{F})=0$, a function $\lambda \in C^{\infty}(\overline{\Omega})$, and a vector field $v$ with components in $C^{\infty}(\overline{\Omega})$ satisfying $v|_{\PD \Omega}=0$ such that  $F$ can be decomposed as 
		\Beq\label{Decomposition formula}
		F= \wt{F} + \lambda g + \D v.
		\Eeq
	
		Here $g$ is the Minkowski metric with $(-1,1,1,\cdots,1)$ along the diagonal and $\D$ is the symmetrized derivative defined by 
		\[
		\lb \D v\rb_{ij}=\frac{1}{2}\lb \PD_{i}v_{j}+\PD_{j} v_{i}\rb.
		\]
	\end{theorem}
	
	See also \cite{Sharafutdinov2007VariationsOD}, where a decomposition result similar in spirit to the one above is shown in a Riemannian setting. 
	
	Combining the above two results, we get the following desired characterization.
	\begin{theorem}\label{Main Theorem 2 for LRT of 2-tensor fields in R3}
		Let $F\in C^{\infty}(\overline{\Omega}; S^2 \Rb^{1+n})$. If for a fixed $\theta_{0}\in \Sb^{n-1}$,
		\begin{align*}
			LF(t,x,\Tl{\theta})=0, \ \mbox{ for all } (t,x)\in\mathbb{R}^{1+n}\ \mbox{ and }\ \theta \mbox{ near } \theta_{0},
		\end{align*}
		then $F=\lambda g + \D v$, where $\lambda\in C^{\infty}(\overline{\Omega})$, $g$ is the Minkowski metric and $v$ is a vector field with components in $C^{\infty}(\overline{\Omega})$ with $v|_{\PD\Omega}=0$. 
	\end{theorem}
	See \cite[Lemma 9.1]{LOSU} as well, where a version of Theorem \ref{Main Theorem 2 for LRT of 2-tensor fields in R3} is proven in the Euclidean setting in space dimensions $n=3$.
	\section{Proofs} 
	
	We prove two lemmas that would immediately give the proof of Theorem \ref{Main Theorem for LRT of 2-tensor fields in R3}. As mentioned already, we will extend $F$ as $0$ outside $\overline{\Omega}$. 
	\begin{lemma}
		Under the assumptions of Theorem \ref{Main Theorem for LRT of 2-tensor fields in R3}, we have 
		the following equality: 
		\[
		\tilde{\theta}^{i}\tilde{\theta}^{j} \wh{F}_{ij}(\zeta)=0 \mbox{ for all } \zeta \in (1,\theta)^{\perp} \mbox{ and } \theta \mbox{ near } \theta_{0}
		\]
		 where $\theta\in \Sb^{n-1}$ and $\perp$ is with respect to the Euclidean metric.
	\end{lemma}
		\begin{proof}
		This is the Fourier slice theorem for the light ray transform. This result is standard;  see \cite{Stefanov_Support_Theorem_Lorentzian_Manifold_2017}. 
		We consider the Fourier transform of $F_{ij}$: 
		\begin{align}
			\widehat{F}_{ij}(\zeta)&=\int\limits_{\mathbb{R}^{1+n}}F_{ij}(t,x)e^{-\I (t,x)\cdot\zeta}\, \D t \D x.
		\end{align}
		Using the decomposition, 
		$\mathbb{R}^{1+n}=\mathbb{R}(1,\theta)\oplus \ell$ with  $\ell\in (1,\theta)^{\perp}$ combined with Fubini's theorem, we get
		\begin{align*}
			\wh{F}_{ij}(\zeta)=\sqrt{2}\int\limits_{(1,\theta)^{\perp}}\int\limits_{\mathbb{R}}F_{ij}(\ell+s(1,\theta))e^{-\I (\ell +s(1,\theta))\cdot\zeta}\, \D s\, \D \ell.
		\end{align*}
		
		If $\zeta\in (1,\theta)^{\perp}$, then 
		\begin{align*}
			\tilde{\theta}^{i}\tilde{\theta}^{j}\widehat{F}_{ij}(\zeta)=\sqrt{2}\int\limits_{(1,\theta)^{\perp}}\int\limits_{\mathbb{R}}{\tilde{{\theta^{i}}}}{\tilde{{\theta^{j}}}}F_{ij}(s(1,\theta)+\ell)e^{-i\ell\cdot\zeta}\, \D s\, \D\ell.
		\end{align*}
		
		Using the fact that 
		\begin{align}\label{vanishing of integration of hijk}
			\int\limits_{\mathbb{R}}\tilde{\theta}^{i}\tilde{\theta}^{j}F_{ij}(t+s,x+s\theta)\, \D s=0, \mbox{ for all }(t,x)\in\mathbb{R}^{1+n},\  \mbox{and}\ \theta\ \text{near}\ \theta_{0},
		\end{align} we  get 
		\begin{align}\label{Vanishing of Fourier Coefficients}
			\tilde{\theta}^{i}\tilde{\theta}^{j}\wh{F}_{ij}(\zeta)=0 \mbox{ for all }\ \zeta\in(1,\theta)^{\perp}\ \text{with}\ \theta\ \text{near}\ \theta_{0}.
		\end{align}
	\end{proof}
	In the following lemma, without loss of generality, we fix $\theta_{0}=(1,0\cdots,0)\in \Sb^{n-1}$.
	\begin{lemma}\label{Main Lemma}
		Let $F\in C^{\infty}(\overline{\Omega}; S^2 \Rb^{1+n})$ be such that $\delta F=0$ and $\mathrm{trace}(F)=0$. Suppose also that $\tilde{\theta}^{i}\tilde{\theta}^{j}\widehat{F}_{ij}(\zeta)=0$ for  $\zeta\in(1,\theta)^{\perp}$ and  $\theta  \text{ near }  \theta_{0}$. 
		Then  
		\[
		\wh{F}_{ij}(\zeta)=0 
		\]
		in a small conical neighborhood of the space-like vector $\zeta_{0}=(0,0,1,0,\cdots,0)\in\Rb^{1+n}$. 
	\end{lemma}
	\bpr
	In order to make the presentation clear, we first give the proof in $\Rb^{1+3}$ and then generalize it to $\Rb^{1+n}$ when $n\geq 4$.

	First let us show that $\wh{F}_{ij}(\zeta_0)=0$ for all $0\leq i,j\leq 3$. We fix $ \theta_{0}=(1,0,0)$. Note that $(1,\theta_{0})\cdot \zeta_{0}=0$. Consider   
	\Beq\label{Perturbation of theta0}
	\theta_{0}(a)=(\cos a,0,\sin a).
	\Eeq If $a$ is near $0$, then $\theta_{0}(a)$ is near $\theta_{0}$. Also note that $(1,\theta_{0}(a))\cdot \zeta_{0}=0$. 
	Substituting $\zeta_{0}=(0,0,1,0)$ and $\theta_{0}(a)$ as above into \eqref{Vanishing of Fourier Coefficients}, we get, 
	\Beq \label{Perturbed eqn}
	\left(\widehat{F}_{00}+ 2\cos{a}\widehat{F}_{01}+ 2\sin a\widehat{F}_{03}+\cos^{2}a\widehat{F}_{11}+2\sin a\cos a\widehat{F}_{13}+\sin^{2}a\widehat{F}_{33}\right)(\zeta_{0})=0.
	\Eeq
	Let us differentiate this equation with respect to $a$ four times. We get,
	\begin{align}
		\label{Once diff equation}&\lb  -2\sin a \wh{F}_{01}+ 2\cos a \wh{F}_{03}-\sin 2a \wh{F}_{11}+ 2\cos 2a \wh{F}_{13}+\sin 2a \wh{F}_{33}\rb({\zeta_{0}})=0\\
		\label{Twice diff equation}&\lb  -2\cos a \wh{F}_{01}- 2\sin a \wh{F}_{03}-2\cos 2a \wh{F}_{11}-4\sin 2a \wh{F}_{13}+2\cos 2a \wh{F}_{33}\rb(\zeta_{0})=0\\
		\label{Three diff equation}&\lb  2\sin a \wh{F}_{01}- 2\cos a \wh{F}_{03}+4\sin 2a \wh{F}_{11}-8\cos 2a \wh{F}_{13}-4\sin 2a \wh{F}_{33}\rb(\zeta_{0})=0\\
		\label{Four diff equation}&\lb  2\cos a \wh{F}_{01}+ 2\sin a \wh{F}_{03}+8\cos 2a \wh{F}_{11}+16\sin 2a \wh{F}_{13}-8\cos 2a \wh{F}_{33}\rb(\zeta_{0})=0.
	\end{align}
	Letting $a\to 0$ in \eqref{Perturbed eqn}, \eqref{Once diff equation}, \eqref{Twice diff equation}, \eqref{Three diff equation} and \eqref{Four diff equation}, we have the following 5 equations: 
	\begin{align}
		\label{a=0, 3DEq1}&\lb\wh{F}_{00} + 2\wh{F}_{01}+\wh{F}_{11}\rb(\zeta_0)=0\\
		\label{a=0, 3DEq2}&\lb  \wh{F}_{03}+ \wh{F}_{13}\rb(\zeta_{0})=0\\
		\label{a=0, 3DEq3}& \lb \wh{F}_{01} +\wh{F}_{11} - \wh{F}_{33}\rb (\zeta_{0})=0\\
		\label{a=0, 3DEq4}&\lb \wh{F}_{03} + 4\wh{F}_{13}\rb (\zeta_{0})=0\\ 
		\label{a=0, 3DEq5}&\lb \wh{F}_{01}+4\wh{F}_{11}-4\wh{F}_{33}\rb (\zeta_{0})=0.
	\end{align}
	
	Since $\delta (F)=\mbox{trace}(F)=0$, we have 
	\begin{align}
		&\label{3Ddelta Eq} \widehat{F}_{02}(\zeta_{0})=\widehat{F}_{12}(\zeta_{0})=\widehat{F}_{22}(\zeta_{0})=\widehat{F}_{32}(\zeta_{0})=0,\\ 
		&\label{3DTrace Eq}\left(\widehat{F}_{00}+\widehat{F}_{11}+\widehat{F}_{22}+\widehat{F}_{33}\right)(\zeta_{0})=0.
	\end{align}
	
	From \eqref{a=0, 3DEq2} and \eqref{a=0, 3DEq4}, we get that $\wh{F}_{03}(\zeta_{0})=\wh{F}_{13}(\zeta_0)=0$. Subtracting \eqref{a=0, 3DEq3} from \eqref{a=0, 3DEq5}, we get that $\wh{F}_{11}(\zeta_0)=\wh{F}_{33}(\zeta_0)$. Therefore \eqref{a=0, 3DEq3} gives that $\wh{F}_{01}(\zeta_0)=0$. Substituting $\wh{F}_{11}(\zeta_0)=\wh{F}_{33}(\zeta_0)$ into \eqref{3DTrace Eq}, and using the fact that $\wh{F}_{22}(\zeta_0)=0$ from \eqref{3Ddelta Eq}, we get that $\wh{F}_{00}(\zeta_0)+2\wh{F}_{11}(\zeta_0)=0$. Combining this with \eqref{a=0, 3DEq1}, we get that $\wh{F}_{00}(\zeta_0)=\wh{F}_{11}(\zeta_0)=0$. Combining all these, we have now shown that $\wh{F}_{ij}(\zeta_0)=0$ for all $0\leq i,j\leq 3$.

    	Next our goal is to show that if $\zeta$ is any non-zero space-like vector in a small enough conical neighborhood (in the Euclidean sense) of $\zeta_{0}$, then {$\wh{F}_{ij}(\zeta)=0$}, for $0\leq i,j\leq 3$ as well. We recall that a non-zero vector $\zeta=(\zeta^{0},\zeta^1,\zeta^2,\zeta^3)$ is space-like if $|\zeta^{0}|<\lVert(\zeta^1,\zeta^2,\zeta^3)\rVert$, where the norm $\lVert \cdot \rVert$ refers to the Euclidean norm.

	We start with a unit vector  $(\zeta^1,\zeta^2,\zeta^3)$ in $\Rb^3$, and we choose $\zeta^0=-\sin \vp$. Then $(-\sin \vp,\zeta^1,\zeta^2,\zeta^3)$ is a space-like vector for $-\pi/2<\vp<\pi/2$.

	Let us recall that in showing  {$\wh{F}_{ij}(\zeta_{0})=0$}, we considered a perturbation $\theta_{0}(a)$ (see  \eqref{Perturbation of theta0}) of the vector $ \theta_{0}=(1,0,0)$. Note that we required that $ \theta_{0}(a)$ was close enough to $\theta_{0}$ and  $(1, \theta_{0}(a))\cdot \zeta_{0}=0$. The following calculations are motivated by having these same requirements for the vector $\zeta=(-\sin \vp,\zeta^1,\zeta^2,\zeta^3)$.
	
	Since we are interested in a non-zero space-like vector in a small enough conical neighborhood of $\zeta_{0}$, let us choose   
	\begin{align*}
		\zeta^{1}=\sin\alpha\cos\beta,\ \zeta^{2}=\cos\alpha \ \text{and}\ \zeta^{3}=\sin\alpha\sin\beta.
	\end{align*}
	Then 
	clearly $\zeta$ is close to $(0,1,0)$ whenever $\alpha$ and $\beta$ are close enough to $0$, and choosing $\vp$ close to $0$, we get that the space-like vector $\zeta=(-\sin \vp,\zeta^1,\zeta^2,\zeta^3)$ is close enough to $(0,0,1,0)$.

	Next choose $\theta_{0}(\vp)= (\cos\vp, \sin\vp,0)$ for $\vp$ close to $0$ and the perturbation  of $\theta_{0}(\vp)$ for $a$ close to $0$ by \[
	\theta_{0}(\vp,a)= \lb  \cos a \cos \vp, \sin\vp,\sin a \cos \vp\rb.
	\]
	Our goal is next to modify $\theta_{0}(\vp,a)$ to ${\Theta}_{0}(\vp,a)$ such that $(1, {\Theta}_{0}(\vp,a))\cdot \zeta=0$. 
	To this end, let us consider the orthogonal matrix $A$: 
	\begin{align*}
		A=
		\begin{bmatrix}
			\cos\alpha\cos\beta& -\sin\alpha& \cos\alpha\sin\beta\\
			\sin\alpha\cos\beta &\cos\alpha& \sin\alpha\sin\beta\\
			-\sin\beta &0&\cos\beta
		\end{bmatrix}
		=
		\begin{bmatrix}
			a_{11}&a_{12}&a_{13}\\a_{21}&a_{22}&a_{23}\\a_{31}&a_{32}&a_{33}
		\end{bmatrix}.
	\end{align*}
	
	Define ${\Theta}_{0}$, ${\Theta}_{0}(\vp)$ and $ {\Theta}_{0}(a,\vp)$ by
	\begin{align*}
		\Theta_{0}= A^{T}
		\begin{bmatrix}
			1\\0\\0
		\end{bmatrix}
		=
		\begin{bmatrix}
			\cos\alpha\cos\beta\\
			-\sin\alpha\\
			\cos\alpha\sin\beta
		\end{bmatrix},
	\end{align*}
	\[
	\Theta_{0}(\vp)= A^T  \theta_{0}(\vp)=A^{T}\begin{bmatrix}
	\cos \vp\\
	\sin \vp\\
	0
	\end{bmatrix}= 
	\begin{bmatrix}
	a_{11}\cos \vp+ a_{21}\sin \vp\\
	a_{12}\cos \vp+a_{22}\sin \vp\\
	a_{13}\cos \vp+a_{23}\sin \vp
	\end{bmatrix}
	\]
	and 
	\[
	\Theta_{0}(\vp,a)=A^{T}\theta_{0}(\vp,a)
	=
	\begin{bmatrix}
	a_{11}\cos a\cos\vp+a_{21}\sin\vp + a_{31}\sin a\cos\vp\\
	a_{12}\cos a\cos\vp+a_{22}\sin\vp+a_{32}\sin a \cos \vp\\
	a_{13}\cos a\cos\vp+a_{23}\sin\vp + a_{33}\sin a\cos\vp
	\end{bmatrix}
	=
	\begin{bmatrix}
	A_{1}(a)\\A_{2}(a)\\A_{3}(a)
	\end{bmatrix}.
	\]
	
	We first note that if $a, \vp, \A$ and $\B$ are close enough to $0$, then $ {\Theta}_{0}(\vp,a)$ is close enough to $ \theta_{0}$. As before, defining $\wt{\Theta}(\vp,a)=(1,{\Theta}(\vp,a))$, we have $LF(t,x, \wt{\Theta}_{0}(\vp,a))$ is $0$. 
	
	Next we show that for all $\vp,a,\A$ and $\B$ close enough to $0, (1, {\Theta}_{0}(\vp,a))\cdot \zeta=0$. To see this, consider
	\begin{align*}
		(-\sin \vp, \sin \A \cos \B,\cos \A,\sin \A\sin \B)\cdot (1, {\Theta}_{0}(\vp))&= - \sin \vp+ \langle \lb \zeta^1,\zeta^2,\zeta^3\rb,A^T \lb  \theta_{0}(\vp,a)\rb\rangle\\
		&=-\sin \vp + \langle  A \lb \zeta^1,\zeta^2,\zeta^3\rb,\theta_{0}(\vp,a)\rangle.
	\end{align*}
	The matrix $A$ is such that $A\lb \zeta^1,\zeta^2,\zeta^3\rb=(0,1,0)$.\\
	Since $\theta_{0}(\vp,a)=( \cos a \cos \vp,\sin \vp,\sin a \cos \vp)$, we now get that $(1, {\Theta}_{0}(\vp,a))\cdot \zeta=0$. Using this choice of ${\Theta}_{0}(\vp,a)$ in \eqref{Vanishing of Fourier Coefficients}, we get 
	
	\begin{align}\label{theta(a) in general case}
		\notag&\Bigg(\widehat{F}_{00}+2A_{1}(a)\widehat{F}_{01}+2A_{2}(a)\widehat{F}_{02}+2A_{3}(a)\widehat{F}_{03}+{\left(A_{1}(a)\right)}^{2}\widehat{F}_{11}+2A_{1}(a)A_{2}(a)\widehat{F}_{12}\\
		&+2A_{1}(a)A_{3}(a)\widehat{F}_{13}+\left(A_{2}(a)\right)^{2}\widehat{F}_{22}+2A_{2}(a)A_{3}(a)\widehat{F}_{23}+\left(A_{3}(a)\right)^{2}\widehat{F}_{33}\Bigg)(\zeta)=0.
	\end{align}
	As before, we consider \eqref{theta(a) in general case} and differentiate it 4 times and let $a\to 0$. These would give 5 equations. Also since $F$ is divergence free and trace free, we have the following 5 equations:  
	\begin{align}\label{equation after using the divergence free and trace free conditions}
		\begin{aligned}
			&\left(-\sin\vp\widehat{F}_{00}+a_{21}\widehat{F}_{01}+a_{22}\widehat{F}_{02}+a_{23}\widehat{F}_{03}\right)(\zeta)=0\\
			&\left(-\sin\vp\widehat{F}_{10}+a_{21}\widehat{F}_{11}+a_{22}\widehat{F}_{12}+a_{23}\widehat{F}_{13}\right)(\zeta)=0\\
			&\left(-\sin\vp\widehat{F}_{20}+a_{21}\widehat{F}_{21}+a_{22}\widehat{F}_{22}+a_{23}\widehat{F}_{23}\right)(\zeta)=0\\
			&\left(-\sin\vp\widehat{F}_{30}+a_{21}\widehat{F}_{31}+a_{22}\widehat{F}_{32}+a_{23}\widehat{F}_{33}\right)(\zeta)=0\\
			& \lb \wh{F}_{00}+\wh{F}_{11}+\wh{F}_{22}+\wh{F}_{33}\rb(\zeta)=0.
		\end{aligned}
	\end{align}
	
	Together, these would give 10 equations and determinant of the matrix formed by the coefficients is continuous as a function of $\alpha,\beta$ and $\varphi$. We show that this determinant 
	 is non-zero, which would give that $\wh{F}_{ij}(\zeta)=0$ for $0\leq i,j\leq 3$. In order to show that the determinant is non-vanishing, it is enough to observe that as $\A,\B,\vp\to 0$ in these 10 equations, 
	we would get the same set of equations as in \eqref{a=0, 3DEq1} - \eqref{3DTrace Eq}. However, we have already shown that {$\wh{F}_{ij}(\zeta_{0})=0$} for $0\leq i,j\leq 3$, using these equations. By continuity of the determinant, we have that the matrix of coefficients formed by the 10 equations mentioned above has non-zero determinant when $\alpha,\beta,\varphi$ close to $0$.  Hence we have $\wh{F}_{ij}(\zeta)=0$ for $0\leq i,j\leq 3$, where $\zeta=(-\sin \vp, \sin \A \cos \B, \cos \A, \sin \A \sin \B)$, with $\A,\B$ and $\vp$ are near $0$. Repeating the same argument as above, we can show that $\wh{F}_{ij}(\lambda \zeta)=0$ for $0\leq i,j\leq 3$, where $\zeta$ is as above and $\lambda >0$. This concludes the Lemma \ref{Main Lemma} for the case of $n=3$. 
	
	Now we consider the general case $n\geq 4$.
	
	As before, first let us show that $\wh{F}_{ij}(\zeta_0)=0$ for all $0\leq i,j\leq n$, where recall that  $\zeta_{0}=(0,0,1,0,\cdots,0)$.  We fix $ \theta_{0}=(1,0,0,\cdots,0)\in\Sb^{n-1}$. Note that $(1,\theta_{0})\cdot \zeta_{0}=0$. Consider   
	\begin{align}\label{Perturbation of theta0 in higher dimensions}
		& \theta_{k}(a)=\cos a e_{1}+\sin a e_{k} \mbox{ for } k\geq 3,\\
		& \theta_{kl}(a)=\cos a e_{1}+\frac{1}{\sqrt{2}}\sin a e_{k}+\frac{1}{\sqrt{2}}\sin a e_{l}  \mbox{ for } 3\leq k<l\leq n
	\end{align}
	where $e_{j}\in\Rb^{n}$ be vector in $\Rb^{n}$ whose $j^{\rm th}$ entry is $1$ and other entries  are zero. If $a$ is near $0$, then $\theta_{0}(a)$ is near $\theta_{0}$. Also note that $(1,\theta_{k}(a))\cdot \zeta_{0}=0$ and $\lb 1,\theta_{kl}(a)\rb\cdot\zeta_{0}=0$. Now substituting this choice of $\zeta_{0}$, $\theta_{k}(a)$ and $\theta_{kl}(a)$  into \eqref{Vanishing of Fourier Coefficients}, we get, 
	\Beq \label{Perturbed eqn in higher dimensions}
	\left(\widehat{F}_{00}+ 2\cos{a}\widehat{F}_{01}+ 2\sin a\widehat{F}_{0k}+\cos^{2}a\widehat{F}_{11}+ 2\sin a\cos a\widehat{F}_{1k}+\sin^{2}a\widehat{F}_{kk}\right)(\zeta_{0})=0\mbox{ for } k\geq 3.
	\Eeq
	\begin{equation}\label{Perturbed equation with theta kl}
		\begin{aligned}
			&\Big(\widehat{F}_{00}+ 2\cos a\widehat{F}_{01}+\sqrt{2}\sin a \widehat{F}_{0k}+\sqrt{2}\sin a\widehat{F}_{0l}+\cos^{2}a\widehat{F}_{11}+ \sqrt{2}\sin a\cos a\widehat{F}_{1k}\\
			&\ \ \ +\sqrt{2}\sin a\cos a\widehat{F}_{1l}+\frac{\sin^{2}a}{2}\lb\widehat{F}_{kk}+2\widehat{F}_{kl}+\widehat{F}_{ll}\rb\Big)(\zeta_{0})=0 \mbox{ for } 3\leq k<l\leq n.
		\end{aligned}
	\end{equation}
	Differentiating \eqref{Perturbed eqn in higher dimensions} 4 times and letting $a\to 0$, we arrive at the following equations: 
	
	\begin{align}
		\label{a=0, NDEq1}&\lb\wh{F}_{00} + 2\wh{F}_{01}+\wh{F}_{11}\rb(\zeta_0)=0\\
		\label{a=0, NDEq2}&\lb  \wh{F}_{0k}+ \wh{F}_{1k}\rb(\zeta_{0})=0\\
		\label{a=0, NDEq3}& \lb \wh{F}_{01} -\wh{F}_{11} + \wh{F}_{kk}\rb (\zeta_{0})=0\\
		\label{a=0, NDEq4}&\lb \wh{F}_{0k} + 4\wh{F}_{1k}\rb (\zeta_{0})=0\\ 
		\label{a=0, NDEq5a}&\lb -\wh{F}_{01}+4\wh{F}_{11}-4\wh{F}_{kk}\rb (\zeta_{0})=0.
	\end{align}

	Similarly, we differentiate \eqref{Perturbed equation with theta kl} with respect to $a$ 4 times and let $a\to 0$. We get,
	
	\begin{align}\label{one wrt theta kl}
		\begin{aligned}
			&\Big(- 2\sin a \widehat{F}_{01}+\sqrt{2}\cos a\widehat{F}_{0k}+\sqrt{2}\cos a\widehat{F}_{0l}-\sin2a\widehat{F}_{11}+\sqrt{2}\cos2a\widehat{F}_{1k}\\
			&+\sqrt{2}\cos2a\widehat{F}_{1l}+\frac{\sin2a}{2}\lb\widehat{F}_{kk}+2\widehat{F}_{kl}+\widehat{F}_{ll}\rb\Big)(\zeta_{0})=0;\ 3\leq k<l\leq n
		\end{aligned}
	\end{align}
	\begin{align}\label{two wrt theta kl}
		\begin{aligned}
			&\Big(- 2\cos a \widehat{F}_{01}-\sqrt{2}\sin a\widehat{F}_{0k}-\sqrt{2}\sin  a\widehat{F}_{0l}-2\cos2a\widehat{F}_{11}-2\sqrt{2}\sin2a\widehat{F}_{1k}\\
			&-2\sqrt{2}\sin2a\widehat{F}_{1l}+\cos2a\lb\widehat{F}_{kk}+2\widehat{F}_{kl}+\widehat{F}_{ll}\rb\Big)(\zeta_{0})=0;\ 3\leq k<l\leq n.
		\end{aligned}
	\end{align}
	
	\begin{align}\label{three wrt theta kl}
		\begin{aligned}
			&\Big(2\sin a \widehat{F}_{01}-\sqrt{2}\cos a\widehat{F}_{0k}-\sqrt{2}\cos  a\widehat{F}_{0l}+4\sin 2a\widehat{F}_{11}-4\sqrt{2}\cos 2a\widehat{F}_{1k}\\
			&-4\sqrt{2}\cos 2a\widehat{F}_{1l}-2\sin 2a\lb\widehat{F}_{kk}+2\widehat{F}_{kl}+\widehat{F}_{ll}\rb\Big)(\zeta_{0})=0;\ 3\leq k<l\leq n.
		\end{aligned}
	\end{align}\begin{align}\label{four wrt theta kl}
		\begin{aligned}
			&\Big(2\cos a \widehat{F}_{01}+\sqrt{2}\sin a\widehat{F}_{0k}+\sqrt{2}\sin  a\widehat{F}_{0l}+8\cos 2a\widehat{F}_{11}+8\sqrt{2}\sin 2a\widehat{F}_{1k}\\
			&+8\sqrt{2}\sin 2a\widehat{F}_{1l}-4\cos 2a\lb\widehat{F}_{kk}+2\widehat{F}_{kl}+\widehat{F}_{ll}\rb\Big)(\zeta_{0})=0;\ 3\leq k<l\leq n.
		\end{aligned}
	\end{align}

	Letting $a\to 0$ in \eqref{one wrt theta kl} - \eqref{four wrt theta kl}, we have,
	\begin{align}
		\label{a=0, NDEq5}&\lb\wh{F}_{0k}+\wh{F}_{0l}+\wh{F}_{1k}+\wh{F}_{1l}\rb(\zeta_0)=0\\
		\label{a=0, NDEq6}&\lb  -2\wh{F}_{01}-2\wh{F}_{11}+\wh{F}_{kk}+2\wh{F}_{kl}+\wh{F}_{ll}\rb(\zeta_{0})=0;\ \ k\geq 3\\
		\label{a=0, NDEq7}&\lb \wh{F}_{0k}+\wh{F}_{0l}+4\wh{F}_{1k}+4\wh{F}_{1l}\rb (\zeta_{0})=0;\ \ k\geq 3.\\
		\label{a=0, NDEq8}&\lb 2\wh{F}_{01}+8\wh{F}_{11}-4(\wh{F}_{kk}+2\widehat{F}_{kl}+\widehat{F}_{ll})\rb(\zeta_{0})=0;\ 3\leq k<l\leq n.
	\end{align}
	
	Now we consider \eqref{a=0, NDEq1} - \eqref{a=0, NDEq4} and \eqref{a=0, NDEq5} - \eqref{a=0, NDEq8} combined with the following two equations: 
	
	\begin{align}
		&\label{delta Eq in the higher dimensions} \widehat{F}_{02}(\zeta_{0})=\widehat{F}_{12}(\zeta_{0})=\widehat{F}_{22}(\zeta_{0})=\cdots =\widehat{F}_{n2}(\zeta_{0})=0;\\ 
		&\label{Trace Eq in the higher dimensions}\left(\widehat{F}_{00}+\widehat{F}_{11}+\widehat{F}_{22}+\widehat{F}_{33}+\cdots+\widehat{F}_{nn}\right)(\zeta_{0})=0,
	\end{align}
	since $\delta (F)=\mbox{trace}(F)=0$. 
	
	We now show that these equations imply that $\wh{F}_{ij}(\zeta_0)=0$ for all $0\leq i,j\leq n$.
	
	Adding \eqref{a=0, NDEq8} and \eqref{a=0, NDEq6}, we get, 
	\Beq\label{Eq1}
	2\wh{F}_{11}-(\wh{F}_{kk}+2\wh{F}_{kl}+\wh{F}_{ll})(\zeta_0)=0 \mbox{ for } 3\leq k<l\neq n.
	\Eeq
	Subtracting \eqref{a=0, NDEq7} from \eqref{a=0, NDEq5}, we get,
	\Beq\label{Eq2}
	(\wh{F}_{1k}+\wh{F}_{1l})(\zeta_0)=(\wh{F}_{0k}+\wh{F}_{0l})(\zeta_0)=0 \mbox{ for } 3\leq k<l\neq n.
	\Eeq
	
	Adding \eqref{a=0, NDEq3} and \eqref{a=0, NDEq5a}, we get,
	\Beq\label{Eq3}
	\wh{F}_{11}(\zeta_0)=\wh{F}_{kk}(\zeta_0) \mbox{ for } k\geq 3 \mbox{ and } \wh{F}_{01}(\zeta_0)=0.
	\Eeq
	Now combined with the previous equation, we have from \eqref{a=0, NDEq1} that \Beq\label{Eq4}
	\wh{F}_{00}(\zeta_0)=-\wh{F}_{11}(\zeta_0).
	\Eeq
	
	From \eqref{a=0, NDEq2} and \eqref{a=0, NDEq4}, we have that 
	\Beq\label{Eq5}
	\wh{F}_{1k}(\zeta_0)=\wh{F}_{0k}(\zeta_0)=0 \mbox{ for } k\geq 3.
	\Eeq
	
	Now we already know from \eqref{delta Eq in the higher dimensions} that $\wh{F}_{22}(\zeta_0)=0$. Using \eqref{Eq3} and \eqref{Eq4} in \eqref{Trace Eq in the higher dimensions}, we get that 
	\Beq\label{Eq6}
	(n-2)\wh{F}_{11}(\zeta_0)=0.
	\Eeq
	This then implies that 
	\Beq\label{Eq7}
	\wh{F}_{mm}(\zeta_0)=0 \mbox{ for all } 0\leq m\leq n.
	\Eeq
	Now from \eqref{Eq1}, this then implies that 
	\Beq\label{Eq8}
	\wh{F}_{kl}(\zeta_0)=0 \mbox{ for all } 3\leq k<l\leq n. 
	\Eeq
	Now combined with \eqref{delta Eq in the higher dimensions}, we now have that \Beq\label{Eq9}
	\wh{F}_{ij}(\zeta_0)=0 \mbox{ for all } 0\leq i,j\leq n.
	\Eeq
		Next our goal is to show that if $\zeta$ is any non-zero space-like vector in a small enough conical neighborhood (in the Euclidean sense) of $\zeta_{0}$, then {$\wh{F}_{ij}(\zeta)=0$}, for $0\leq i,j\leq n$ as well. We recall that a non-zero vector $\zeta=(\zeta^{0},\zeta^1,\zeta^2,\cdots,\zeta^n)$ is space-like if $|\zeta^{0}|<\lVert(\zeta^1,\zeta^2,\cdots, \zeta^n)\rVert$, where the norm $\lVert \cdot \rVert$ refers to the Euclidean norm.

	We start with a unit vector in $\Rb^n$, $\zeta':=(\zeta^1,\zeta^2,\cdots,\zeta^n)$, and let us choose $\zeta^0=-\sin \vp$. Then $(-\sin \vp,\zeta^1,\zeta^2,\cdots,\zeta^n)$ is a space-like vector if $-\pi/2<\vp<\pi/2$.

	Let us recall that in showing  $\wh{F}_{ij}(\zeta_{0})=0$, we considered a perturbation $\theta_{0}(a)$ (see  \eqref{Perturbation of theta0 in higher dimensions}) of the vector $\theta_{0}=(1,0,\cdots,0)$. Note that we required that $\theta_{0}(a)$ was close enough to $\theta_{0}$ and  $(1, \theta_{0}(a))\cdot \zeta_{0}=0$. As in the proof for the case $n=3$, the calculations below are motivated by these requirements for the vector $\zeta=(-\sin \vp,\zeta^1,\zeta^2,\cdots,\zeta^n)$.
	
	Since we are interested in a non-zero space-like vector in a small enough conical neighborhood of $\zeta_{0}$, let us choose $\zeta'$ as   
	\[
	\zeta'=(\cos \vp_1 \sin \vp_2, \cos \vp_2, \sin \vp_1 \sin \vp_2\cos \vp_3,\cdots, \sin\vp_{1}\sin\vp_{2}\cdots\sin\vp_{n-2}\sin\vp_{n-1}).
	\]
		Then 
	clearly $\zeta'$ is close to $(0,1,0,\cdots,0)\in\Rb^{n}$ whenever $\vp_{i}$  for $1\leq i\leq n-1$ are close enough to $0$, and choosing $\vp$ close to $0$, we get that the space-like vector $\zeta=(-\sin \vp,\zeta^1,\zeta^2,\cdots,\zeta^n)$ is close enough to $\zeta_{0}=(0,0,1,0,\cdots,0)\in \Rb^{1+n}$.

	Next choose $\theta_{0}(\vp):= \cos\vp e_{1}+\sin\vp e_{2}$ close to $\theta_{0}$ when $\vp$ is close to $0$ and the perturbation  of $\theta_{0}(\vp)$ for $a$ close to $0$ by 
	\begin{align*}
		\begin{aligned}
			&\theta_{k}(\vp,a)=   \cos a \cos \vp e_{1}+ \sin\vp e_{2}+\sin a \cos \vp e_{k} \mbox{ for } k\geq 3,\\
			&\theta_{kl}(\vp,a)= \cos a \cos \vp e_{1}+ \sin\vp e_{2}+\frac{1}{\sqrt{2}}\sin a \cos \vp e_{k}+\frac{1}{\sqrt{2}}\sin a \cos \vp e_{l} \mbox { for } 3\leq k<l\leq n.
		\end{aligned}
	\end{align*}
	Let us consider the orthogonal matrix $A$ such that $A\zeta'=e_{2}$. Let us denote the entries of this matrix by $A=(a_{ij})$.
		Define ${\Theta}_{0}(\vp)$ and  ${\Theta}_{k}(a,\vp)$ and ${\Theta}_{kl}(\vp,a)$ by
	
	\begin{align*}
		\begin{aligned}
			{\Theta}_{k}(\vp,a)=A^{T}\lb \theta_{k}(\vp,a)\rb
			=
			\begin{bmatrix}
				a_{11}\cos a\cos\vp+a_{21}\sin\vp + a_{k1}\sin a\cos\vp\\
				\vdots\\
				a_{1n}\cos a\cos\vp+a_{2n}\sin\vp + a_{kn}\sin a\cos\vp
			\end{bmatrix}
			=
			\begin{bmatrix}
				A_{1}(a)\\\vdots\\A_{n}(a)
			\end{bmatrix}
			\mbox{where $k\geq 3$}
		\end{aligned}
	\end{align*}
	and 
	\begin{align*}
		\begin{aligned}
			{\Theta}_{kl}(\vp,a)=A^{T}\lb \theta_{kl}(\vp,a)\rb=
			\begin{bmatrix}
				a_{11}\cos a\cos\vp+a_{21}\sin\vp + \frac{1}{\sqrt{2}}\lb a_{k1}+a_{l1}\rb \sin a\cos\vp\\
				\vdots\\
				a_{1n}\cos a\cos\vp+a_{2n}\sin\vp + \frac{1}{\sqrt{2}}\lb a_{kn}+a_{ln}\rb \sin a\cos\vp
			\end{bmatrix}
			=
			\begin{bmatrix}
				B_{1}(a)\\\vdots\\B_{n}(a)
			\end{bmatrix},
		\end{aligned}
	\end{align*}
	where $3\leq k<l\leq n$.
	We first note that if $a, \vp,$ and $\vp_{i}$ for $1\leq i\leq n-1$,  are close enough to $0$, then ${\Theta}_{k}(\vp,a)$ and ${\Theta}_{kl}(\vp,a)$ are close enough to $\theta_{0}$. Denoting $\wt{\Theta}_{k}(\vp,a)=(1,\Theta_{k}(\vp,a)$ and $\wt{\Theta}_{k,l}(\vp,a)=(1,\Theta_{k,l}(\vp,a)$, we have that $LF(t,x, \wt{\Theta}_{k}(\vp,a))=0$ for $k\geq 3$ and $LF(t,x, \wt{\Theta}_{kl}(\vp,a))=0$ for $3\leq k<l\leq n$.
	
	Next we show that for all $\vp,a$ and $\vp_{i}$ for $1\leq i\leq n-1$ close enough to $0, (1, \wt{\Theta}_{k}(\vp,a))\cdot \zeta=0$ and $\lb 1,\wt{\Theta}_{kl}(\vp,a)\rb\cdot\zeta=0$. 
	
	To see this, consider
	\begin{align*}
		(-\sin \vp, \zeta')\cdot (1, {\Theta}_{k}(\vp,a))= -\sin\vp + \langle \zeta',{\Theta}_{k}(\vp)\rangle &=-\sin\vp + \langle \zeta',A^{T}\lb  {\theta}_{k}(\vp,a)\rb\rangle \\
		&= -\sin\vp +\langle A\zeta', \theta_{k}(\vp,a)\rangle.
	\end{align*}
	Note that the matrix $A$ is chosen such that  $A\lb \zeta'\rb=(0,1,0,\cdots,0)$.
	Since $\theta_{k}(\vp,a)=\cos a \cos \vp e_{1}+\sin \vp e_{2}+\sin a \cos \vp e_{k}$, $k\geq 3$, we now get that $(1,{\Theta}_{k}(\vp,a))\cdot \zeta=0$. Similarly we can check that $\lb 1,\Theta_{kl}(\vp,a)\rb \cdot\zeta=0$. 
	
	Using this choice of $\wt{\Theta}_{k}(\vp,a)$ in \eqref{Vanishing of Fourier Coefficients}, we have
	
	\begin{align}\label{equation after using pm theta(a) in general case in higher dimensions}
		\begin{aligned}
			&\Bigg(\widehat{F}_{00}+2A_{1}(a)\widehat{F}_{01}+2A_{2}(a)\widehat{F}_{02}+ 2A_{3}(a)\widehat{F}_{03}+ \cdots+2A_{n}(a)\widehat{F}_{0n}\\
			&+{\left(A_{1}(a)\right)}^{2}\widehat{F}_{11}+2A_{1}(a)A_{2}(a)\widehat{F}_{12}+ 2A_{1}(a)A_{3}(a)\widehat{F}_{13}+\cdots+2A_{1}(a)A_{n}(a)\widehat{F}_{1n}\\
			&+{\lb A_{2}(a)\rb^{2}}\widehat{F}_{22}+ 2A_{2}(a)A_{3}(a)\widehat{F}_{23}+2A_{2}(a)A_{4}(a)\widehat{F}_{24}+\cdots+2A_{2}(a)A_{n}(a)\widehat{F}_{2n}\\
			& \hspace{2.5in}\vdots\\
			&+\lb A_{n-1}(a)\rb^{2}\widehat{F}_{n-1,n-1}+ 2A_{n-1}(a)A_{n}(a)\widehat{F}_{n-1,n}+\left(A_{n}(a)\right)^{2}\widehat{F}_{nn}\Bigg)(\zeta)=0.
		\end{aligned}
	\end{align}
	
	Next using the choice $\wt{\Theta}_{kl}(\vp,a)$ in  \eqref{Vanishing of Fourier Coefficients}, we get
	\begin{align}\label{equation after using pm thetakl(a) in general case in higher dimensions}
		\begin{aligned}
			&\Bigg(\widehat{F}_{00}+2B_{1}(a)\widehat{F}_{01}+2B_{2}(a)\widehat{F}_{02}+2B_{3}(a)\widehat{F}_{03}+\cdots+2B_{n}(a)\widehat{F}_{0n}\\
			&+{\left(B_{1}(a)\right)}^{2}\widehat{F}_{11}+2B_{1}(a)B_{2}(a)\widehat{F}_{12}+2B_{1}(a)B_{3}(a)\widehat{F}_{13}+\cdots+2B_{1}(a)B_{n}(a)\widehat{F}_{1n}\\
			&+{\lb B_{2}(a)\rb^{2}}\widehat{F}_{22}+2B_{2}(a)B_{3}(a)\widehat{F}_{23}+2B_{2}(a)B_{4}(a)\widehat{F}_{24}+\cdots+2B_{2}(a)B_{n}(a)\widehat{F}_{2n}\\
			&\hspace{2.5in}\vdots\\
			&+\lb B_{n-1}(a)\rb^{2}\widehat{F}_{n-1,n-1}+2B_{n-1}(a)B_{n}(a)\widehat{F}_{n-1,n}+\left(B_{n}(a)\right)^{2}\widehat{F}_{nn}\Bigg)(\zeta)=0.
		\end{aligned}
	\end{align}
	We differentiate each of Equations \eqref{equation after using pm theta(a) in general case in higher dimensions} and \eqref{equation after using pm thetakl(a) in general case in higher dimensions}, 4 times and let $a\to 0$. Arguing similarly to the case of $n=3$, we will arrive at the fact that $\wh{F}_{ij}(\zeta)=0$, and also $\wh{F}_{ij}(\lambda \zeta)=0$ for $\lambda >0$.
	
\end{proof}
\begin{proof}[Proof of Theorem \ref{Main Theorem for LRT of 2-tensor fields in R3}] By Lemma \ref{Main Lemma}, we have an open cone of space-like vectors $\zeta$ along which the Fourier transform $\wh{F}_{ij}(\zeta)=0$. Since $F_{ij}$ for $1\leq i,j\leq n$ are extended by zero outside $\Omega$, therefore using  Paley-Wiener theorem, we have that $F_{ij}\equiv 0$ for all $0\leq i,j\leq n$.
\end{proof}

Next we prove the decomposition result stated in Theorem \ref{DT}. 
\begin{proof}[Proof of Theorem \ref{DT}]
	Assume the decomposition is true. Taking trace on both sides in \eqref{Decomposition formula}, we get,
	\[
	\mbox{trace}(F)=\mbox{trace}(\wt{F})+ \mbox{trace}(\lambda g) + \mbox{trace}(\D v).
	\]
	Now by assumption, $\mbox{trace}(\wt{F})=0$ and $\mbox{trace}(\lambda g) = (n-1)\lambda$ . Also $\mbox{trace}(\D v)= \delta v$. Therefore
	\Beq\label{Trace equation}
	\mbox{trace}(F)=(n-1)\lambda + \delta v.
	\Eeq
	Let us take divergence on both sides of \eqref{Decomposition formula}. Using the fact that $\wt{F}$ is divergence free
	\[
	\delta F= \delta \lb \lambda g\rb + \delta \D v.
	\]
	
	Writing the above  equation in expanded form, we have 
	\begin{align}\label{Equation after taking divergence}
		\begin{aligned}
			\begin{bmatrix}\vspace{2mm}
				&\PD_{j}F_{0j}\\\vspace{2mm}
				&\PD_{j}F_{1j}\\\vspace{2mm}
				&\vdots\\\vspace{2mm}
				&\PD_{j}F_{nj}
			\end{bmatrix}
			=
			\begin{bmatrix}\vspace{2mm}
				-\PD_{0}\lambda\\\vspace{2mm}
				\PD_{1}\lambda\\\vspace{2mm}
				\vdots\\\vspace{2mm}
				\PD_{n}\lambda
			\end{bmatrix}
			+\frac{1}{2}
			\begin{bmatrix}\vspace{2mm}
				&\Delta v_{0}+\PD_{0j}^{2}v_{j}\\\vspace{2mm}
				&\Delta v_{1}+\PD_{1j}^{2}v_{j}\\\vspace{2mm}
				&\vdots\\\vspace{2mm}
				&\Delta v_{n}+\PD_{nj}^{2}v_{j}
			\end{bmatrix}
			.
		\end{aligned}
	\end{align}
	
	Now using the expression for $\lambda$ from \eqref{Trace equation} in \eqref{Equation after taking divergence}, we get 
	\begin{align*}
		\begin{aligned}
			\frac{1}{2}
			\begin{bmatrix}\vspace{2mm}
				&\Delta v_{0}+\PD_{0j}^{2}v_{j}\\\vspace{2mm}
				&\Delta v_{1}+\PD_{1j}^{2}v_{j}\\\vspace{2mm}
				&\vdots\\\vspace{2mm}
				&\Delta v_{n}+\PD_{nj}^{2}v_{j}
			\end{bmatrix}
			+\frac{1}{n-1}
			\begin{bmatrix}\vspace{2mm}
				&-\PD_{0} \mbox{trace}(F)\\\vspace{2mm}
				&\PD_{1} \mbox{trace}(F)\\\vspace{2mm}
				&\vdots\\\vspace{2mm}
				&\PD_{n} \mbox{trace}(F)
			\end{bmatrix}
			-\frac{1}{n-1}
			\begin{bmatrix}\vspace{2mm}
				&-\PD^{2}_{0j}v_{j}\\\vspace{2mm}
				&\PD^{2}_{1j}v_{j}\\\vspace{2mm}
				&\vdots\\\vspace{2mm}
				&\PD^{2}_{nj}v_{j}
			\end{bmatrix}
			=
			\begin{bmatrix}\vspace{2mm}
				&\PD_{j}F_{0j}\\\vspace{2mm}
				&\PD_{j}F_{1j}\\\vspace{2mm}
				&\vdots\\\vspace{2mm}
				&\PD_{j}F_{nj}
			\end{bmatrix}
			.
		\end{aligned}
	\end{align*}
	Thus the equation for $v$ is 
	\begin{align}\label{Equation for v general case}
		\begin{aligned}
			\begin{bmatrix}\vspace{2mm}
				&\Delta v_{0}+\lb 1+\frac{2}{n-1}\rb \PD_{0j}^{2}v_{j}\\\vspace{2mm}
				&\Delta v_{1}+\lb 1-\frac{2}{n-1}\rb \PD_{1j}^{2}v_{j}\\\vspace{2mm}
				&\vdots\\\vspace{2mm}
				&\Delta v_{n}+\lb 1-\frac{2}{n-1}\rb \PD_{nj}^{2}v_{j}
			\end{bmatrix}
			=2
			\begin{bmatrix}\vspace{2mm}
				&\PD_{j}F_{0j}\\\vspace{2mm}
				&\PD_{j}F_{1j}\\\vspace{2mm}
				&\vdots\\\vspace{2mm}
				&\PD_{j}F_{nj}
			\end{bmatrix}
			-\frac{2}{n-1}
			\begin{bmatrix}\vspace{2mm}
				&-\PD_{0} \mbox{trace}(F)\\\vspace{2mm}
				&\PD_{1} \mbox{trace}(F)\\\vspace{2mm}
				&\vdots\\\vspace{2mm}
				&\PD_{n} \mbox{trace}(F)
			\end{bmatrix}
			:=
			\begin{bmatrix}\vspace{2mm}
				u_{0}\\\vspace{2mm}
				u_{1}\\\vspace{2mm}
				\vdots\\\vspace{2mm}
				u_{n}
			\end{bmatrix}
			.
		\end{aligned}
	\end{align}
	This set of equations can be written as
	\begin{align}\label{Equation for vj}
		\begin{aligned}
			&\Delta v_{0}+\lb 1+\frac{2}{n-1}\rb \PD_{0k}^{2}v_{k}=u_{0}\\
			&\Delta v_{j}+\lb 1-\frac{2}{n-1}\rb \PD_{jk}^{2}v_{k}=u_{j}, \ \mbox{for $1\leq j\leq n$}.
		\end{aligned}
	\end{align}
	We first note that for $n=3$, the above system of equations becomes
	
	\begin{align}\label{Equation for vj n equal 3}
		\begin{aligned}
			\begin{cases}
				3\PD_{0}^{2}v_{0}+\PD_{1}^{2}v_{0}+\PD_{2}^{2}v_{0}+\PD_{3}^{2}v_{0}+2\lb \PD_{01}^{2}v_{1}+\PD_{02}^{2}v_{2}+\PD_{03}^{2}v_{3}\rb =u_{0},\ \mbox{in $\Omega$}\\
				\Delta v_{1}=u_{1}, \mbox{in $\Omega$}\\
				\Delta v_{2}=u_{2}, \ \mbox{in $\Omega$}\\
				\Delta v_{3}=u_{3}, \ \mbox{in $\Omega$}\\
				v_{0}=v_{1}=v_{2}=v_{3}=0,\ \mbox{on  $\PD\Omega$}.
			\end{cases}
		\end{aligned}
	\end{align}
	Equation \eqref{Equation for vj n equal 3}  is a decoupled system of equations for $v$ with zero Dirichlet boundary data and hence it is uniquely solvable. Then we use \eqref{Trace equation} to solve for $\lambda$. This completes the proof of Theorem \ref{DT} for $n=3$.
	
	Now in what follows, we assume that $n\geq 4$.

	For simplicity, we denote $\alpha=1+\frac{2}{n-1}$, $\beta= 1-\frac{2}{n-1}$ and  $A({t},x;{\n})$ the following operator (here and below $\n=(\PD_t,\PD_{x_{1}},\cdots,\PD_{x_{n}}$)): 
		\begin{align}\label{Operator A}
			A({t}, x;{\n})=
			\begin{bmatrix}
				&\Delta+\alpha\PD_{0}^{2} & \alpha\PD_{01}^{2}& \alpha\PD_{02}^{{2}}&\cdots&\alpha\PD_{0n}^{2}\\
				&\beta\PD_{10}^{2}&\Delta+\beta\PD_{1}^{2}&\beta\PD_{12}^{2}&\cdots&\beta\PD_{1n}^{2}\\
				&\beta\PD_{20}^{2}&\beta\PD_{21}^{2}&\Delta+\beta\PD_{2}^{2}&\cdots&\beta\PD_{2n}^{2}\\
				&\vdots&\vdots&\vdots&\ddots&\vdots\\
				&\beta\PD_{n0}^{2}&\beta\PD_{n1}^{2}&\beta\PD_{n2}^{2}&\cdots&\Delta+\beta\PD_{n}^{2}
			\end{bmatrix}.
				\end{align}
	Then we have \eqref{Equation for vj} with the homogeneous boundary condition can be written as
	\begin{align}\label{BVP in compact form}
		\begin{aligned}
			\begin{cases}
				A({t},x;{\n})v({t},x)=u({t},x) \quad ({t},x)\in\Omega,\\
				v({t},x)=0 \quad {(t,x)}\in\PD\Omega
			\end{cases}
		\end{aligned}
	\end{align}
	where $v({t}, x)=\lb v_{0}({t},x),v_{1}({t},x),\cdots,v_{n}({t},x)\rb^{T}$ and $u({t},x)=\lb u_{0}({t},x),u_{1}({t},x),\cdots,u_{n}({t},x)\rb^{T}$ are two column vectors. Our goal is to show that the  boundary value problem \eqref{BVP in compact form} is uniquely solvable.
	To this end, we show (see \cite{Sharaf_Book,Taylor_Book}) that 
	$A({t},x;{\n})$  is  strongly elliptic  with zero kernel and zero co-kernel.

	We first prove strong ellipticity. The symbol $A({t,x};\xi)$ of operator $A({t}, x;{\n})$ is  (up to a sign) given by 
	\begin{align}\label{Symbol of A}
		\begin{aligned}
			A({t}, x;\xi)=
			\begin{bmatrix}
				&\lvert\xi\rvert^{2}+\alpha\xi_{0}^{2} & \alpha\xi_{0}\xi_{1}& \alpha\xi_{0}\xi_{2}&\cdots&\alpha\xi_{0}\xi_{n}\\
				&\beta\xi_{1}\xi_{0}&\lvert\xi\rvert^{2}+\beta\xi_{1}^{2}&\beta\xi_{1}\xi_{2}&\cdots&\beta\xi_{1}\xi_{n}\\
				&\beta\xi_{2}\xi_{0}&\beta\xi_{2}\xi_{1}&\lvert\xi\rvert^{2}+\beta\xi_{2}^{2}&\cdots&\beta\xi_{2}\xi_{n}\\
				&\vdots&\vdots&\vdots&\ddots&\vdots\\
				&\beta\xi_{n}\xi_{0}&\beta\xi_{n}\xi_{1}&\beta\xi_{n}\xi_{2}&\cdots&\lvert\xi\rvert^{2}+\beta\xi_{n}^{2},
			\end{bmatrix}
					\end{aligned}
	\end{align}
	where $\xi=(\xi_0,\xi_1,\cdots,\xi_n)$. To prove strong ellipticity for $A({t},x;{\n})$ it is enough to show that  $$P({t}, x;\xi) =\frac{A({t}, x;\xi)+A^{T}({t}, x;\xi)}{2}$$ 
	is positive definite. Now
		\begin{align*}
			P({t}, x;\xi)=
			\begin{bmatrix}
				&\lvert\xi\rvert^{2}+\alpha\xi_{0}^{2} & \xi_{0}\xi_{1}& \xi_{0}\xi_{2}&\cdots&\xi_{0}\xi_{n}\\
				&\xi_{1}\xi_{0}&\lvert\xi\rvert^{2}+\beta\xi_{1}^{2}&\beta\xi_{1}\xi_{2}&\cdots&\beta\xi_{1}\xi_{n}\\
				&\xi_{2}\xi_{0}&\beta\xi_{2}\xi_{1}&\lvert\xi\rvert^{2}+\beta\xi_{2}^{2}&\cdots&\beta\xi_{2}\xi_{n}\\
				&\vdots&\vdots&\vdots&\ddots&\vdots\\
				&\xi_{n}\xi_{0}&\beta\xi_{n}\xi_{1}&\beta\xi_{n}\xi_{2}&\cdots&\lvert\xi\rvert^{2}+\beta\xi_{n}^{2}
			\end{bmatrix}
			.
	\end{align*}
	Let $\eta\in\Rb^{1+n}\setminus{\{0\}}$. Then $\eta^{T}P({t}, x;\xi)\eta$  is given by 
	
	\begin{align*}
		\begin{aligned}
			\eta^{T}P({t}, x;\xi)\eta&=\lvert\xi\rvert^{2}\lvert\eta\rvert^{2}+\lb\alpha-1\rb \xi_{0}^{2}\eta_{0}^{2}+\xi_{0}\eta_{0}\lb \xi\cdot\eta\rb+(1-\beta)\xi_{0}\eta_{0}\lb \xi\cdot\eta-\xi_{0}\eta_{0}\rb+\beta\xi\cdot\eta\lb \xi\cdot\eta-\xi_{0}\eta_{0}\rb\\
			&=\lvert\xi\rvert^{2}\lvert\eta\rvert^{2}+\lb \alpha+\beta-2\rb \xi_{0}^{2}\eta_{0}^{2}+\beta\lb \xi\cdot\eta\rb^{2}+2\lb 1-\beta\rb \lb\xi\cdot\eta\rb\xi_{0}\eta_{0}.\\
		\end{aligned}
	\end{align*} 
	Now using the value of $\alpha$ and $\beta$, we have 
	\begin{align*}
		\begin{aligned}
			\eta^{T}P(t,x;\xi)\eta&=\lvert\xi\rvert^{2}\lvert\eta\rvert^{2}+\frac{n-3}{n-1}\lb \xi\cdot\eta\rb^{2}+\frac{4}{n-1}\lb \xi_{0}\eta_{0}\rb\lb\xi\cdot\eta\rb\\
			&=\frac{\lvert\xi\rvert^{2}\lvert\eta\rvert^{2}}{n-1}\lb n-1+(n-3)\lb \frac{\xi\cdot\eta}{\lvert\xi\rvert\lvert\eta\rvert}\rb^{2}+4\lb\frac{\xi_{0}\eta_{0}}{\lvert\xi\rvert\lvert\eta\rvert}\rb\lb \frac{\xi\cdot\eta}{\lvert\xi\rvert\lvert\eta\rvert}\rb\rb.
		\end{aligned}
	\end{align*}
	Let us write the vectors $\xi$ and $\eta$ as $\xi=(\xi_{0},\xi')$ and $\eta=(\eta_{0},\eta')$. Now, for simplicity, we define $A=\frac{\xi_{0}\eta_{0}}{\lvert\xi\rvert\lvert\eta\rvert}$ and $B=\frac{\xi'\cdot\eta'}{\lvert\xi\rvert\lvert\eta\rvert}$, then clearly $\lvert A\rvert\leq 1$, $\lvert B\rvert \leq 1$ and $\lvert A+B\rvert\leq 1$. Using these in the above equation, we have 
	\begin{align*}
		\eta^{T}P({t}, x;\xi)\eta&=\frac{\lvert\xi\rvert^{2}\lvert\eta\rvert^{2}}{n-1}\lb n-1+(n-3)\lb A+B\rb^{2}+4A\lb A+B\rb\rb\\
		&=\frac{\lvert\xi\rvert^{2}\lvert\eta\rvert^{2}}{n-1}\lb n-1+\lb n+1\rb A^{2}+2(n-1)AB+(n-3)B^{2}\rb\\
		&\geq\frac{\lvert\xi\rvert^{2}\lvert\eta\rvert^{2}}{n-1}\lb n-1+(n+1)A^{2}-(n-1)A^{2}-(n-1)B^{2}+(n-3)B^{2}\rb\\
		&\geq \frac{\lvert\xi\rvert^{2}\lvert\eta\rvert^{2}}{n-1}\lb n-1+2A^{2}-2B^{2}\rb\geq \frac{n-3}{n-1}\lvert\xi\rvert^{2}\lvert\eta\rvert^{2}.
	\end{align*}
	This proves that $P({t}, x,\xi)$ is positive definite and hence $A({t}, x;{\n})$ is strongly elliptic for $n\geq 4$.\\
	
	Next we show that \eqref{BVP in compact form} with $u=0$ on the right hand side has only the zero solution. Multiplying the first equation in \eqref{Equation for vj}  by $v_{0}$ and second equation in \eqref{Equation for vj} by $v_{j}$ and integrating over $\Omega$, we get the following set of equations 
	\begin{align}\label{Equation for v0 after integration}
		\int\limits_{\Omega}\lvert \n v_{0}({t}, x)\rvert^{2}\, \D {t}\D x+\lb 1+\frac{2}{n-1}\rb \int\limits_{\Omega}\nabla\cdot v({t}, x)\PD_{0}v_{0}({t},x)\, \D {t}\D x=0
	\end{align}
	and for $1\leq j\leq n$
	\begin{align}\label{Equation for vj after integration}
		\int\limits_{\Omega}\lvert \n v_{j}(t,x)\rvert^{2}\, \D {t}\D x+\lb 1-\frac{2}{n-1}\rb \int\limits_{\Omega}\nabla\cdot v({t},x)\PD_{j}v_{j}({t}, x)\, \D {t}\D x=0. 
	\end{align}
	Adding the set of equations in \eqref{Equation for v0 after integration} and \eqref{Equation for vj after integration}, we get 
	\begin{align}\label{Quadratic in div v}
		\begin{aligned}
			\int\limits_{\Omega}\sum_{j=0}^{n}\lvert \nabla v_{j}({t},x)\rvert^{2}\, \D {t}\D x+\lb 1-\frac{2}{n-1}\rb \int\limits_{\Omega}\lvert \nabla\cdot v({t}, x)\rvert^{2}\, \D {t}\D x+\frac{4}{n-1}\int\limits_{\Omega}\nabla\cdot v({t}, x)\PD_{0} v_{0}({t}, x)\, \D {t}\D x=0. 
		\end{aligned}
	\end{align}
		For simplicity, let us denote $a=\PD_{0}v_{0}$, $b=\sum_{j=1}^{n}\PD_{j}v_{j}$ and $c= \sum_{j=0}^{n}\lvert\nabla v_{j}\rvert^{2}-\lvert\PD_{0}v_{0}\rvert^{2}$. Using these in \eqref{Quadratic in div v}, we have 
	\begin{align*}
		\begin{aligned}
			\int\limits_{\Omega} \lb c+a^{2}+\frac{n-3}{n-1}\lb a+b\rb^{2}+\frac{4}{n-1}\lb a^{2}+ab\rb\rb \D {t} \D x=0.
		\end{aligned}
	\end{align*}
	Rewriting this, we get,  
	\begin{align*}
		\begin{aligned}
			\int\limits_{\Omega} \lb 2n a^{2}+2\lb n-1\rb ab+\lb n-3\rb b^{2}+\lb n-1\rb c\rb \D {t}\D x=0.
		\end{aligned}
	\end{align*}
	Now let us view the integrand in the above equation as a quadratic equation in $a$ and its discriminant $D({t},x)$ is given by 
	\begin{equation*}
		\begin{aligned}
			D({t},x)&=4(n-1)^{2}b^{2}- 8n \lb (n-3)b^{2}+(n-1)c\rb\\
			&=4\Big{(}\lb n^{2}-2n+1\rb b^{2}-2n(n-3)b^{2}-2n(n-1)c\Big{)}\\
			&= 4\Big{(}\lb -n^{2}+4n+1\rb b^{2}-2n(n-1)c\Big{)}.
		\end{aligned}
	\end{equation*}
{ Now 
\begin{align*}
c=\sum_{j=0}^{n}\lvert \nabla v_{j}\rvert^{2}-\lvert\PD_{0}v_{0}\rvert^{2}
=\sum_{i,j=0}^{n}|\PD_{i} v_{j}|^{2}-\lvert\PD_{0}v_{0}\rvert^{2}
\geq \sum_{j=1}^{n}\lvert \PD_{j}v_{j}\rvert^{2}.
\end{align*}
Also 
\[b^{2}=\big\lvert\sum_{j=1}^{n}\PD_{j}v_{j}\big\rvert^{2}=\sum_{j=1}^{n}\lvert\PD_{j}v_{j}\rvert^{2}+2\sum_{1\leq j<k\leq n}\mbox{Re}\lb \PD_{j}v_{j}\overline{\PD_{k}v_{k}}\rb \leq n\sum_{j=1}^{n}\lvert\PD_{j}v_{j}\rvert^{2}\leq nc. \]
}
	Thus we have that $nc\geq b^{2}$ and using this we get 
	\begin{align*}
		D({t},x)\leq 4\lb -n^{2}+4n+1 -2n+2\rb b^{2}=4\lb -n^{2}+2n +3\rb b^{2}<0 \mbox{ if } b^{2}\neq 0 \mbox{ and } n\geq 4.
	\end{align*}
	However if $D({t},x)<0$, we have the integrand in \eqref{Quadratic in div v} is strictly positive which is not possible since the integral in \eqref{Quadratic in div v} is zero. Hence we have $b=0$ and using this in \eqref{Quadratic in div v}, we have $\sum_{j=0}^{n}\lvert\nabla v_{j}\rvert^{2}=0$ in $\Omega$. This implies $v_{j}({t},x)=c_{j}$ for $0\leq j\leq n$ where $c_{j}$ is some constant. Now using the boundary condition we have that $v_{j}({t},x)=0$ in $\Omega$. Hence Ker$\lb A({t},x;{\n})\rb ={\{0\}}$.\\

	Finally, we show that the co-kernel of $A({t},x;{\n})$ is $0$ as well. 
	We proceed as follows.  Let $w\in \lb \mbox{Image}(A({t},x;{\n}))\rb^{\perp}$. That is, consider $w$ such that  
	\Beq\label{cokernel equation}
	\langle w,A({t},x;{\n})v\rangle=0 \mbox{ for all } v\in C^{\infty}(\Omega) \mbox{ with } v=0 \mbox{ on } \PD \Omega.
	\Eeq
	This, in particular, gives  \Beq\label{cokernel equation 1}
	\langle A^{*}({t},x;{\n})w,v\rangle=0 \mbox{ for all } v\in C^{\infty}_{c}(\Omega),
	\Eeq
	where 
	\begin{align}\label{Adjoint to A}
		\begin{aligned}
			A^{*}({t},x;\PD)=
			\begin{bmatrix}
				&\Delta+\alpha\PD_{0}^{2}&\beta\PD_{10}^{2}&\beta\PD_{20}^{2}&\cdots&\beta\PD_{n0}^{2}\\
				&\alpha\PD_{01}^{2}&\Delta+\beta\PD_{1}^{2}&\beta\PD_{21}^{2}&\cdots&\beta\PD_{n1}^{2}\\
				&\alpha\PD_{02}^{2}&\beta\PD_{12}^{2}&\Delta+\beta\PD_{2}^{2}&\cdots&\beta\PD_{n2}^{2}\\
				&\vdots&\vdots&\vdots&\ddots&\vdots\\
				&\alpha\PD_{0n}^{2}&\beta\PD_{1n}^{2}&\beta\PD_{2n}^{2}&\cdots&\Delta+\beta\PD_{n}^{2}
			\end{bmatrix}
			.
		\end{aligned}
	\end{align}	
	By integration by parts in \eqref{cokernel equation},  combined with \eqref{Operator A} and the fact that $v|_{\PD\Omega}=0$, we have 
	\begin{align}\label{identity for showing vanishing of w on boundary}
		\begin{aligned} 
			0=\langle w,A({t},x,{\n})v\rangle_{L^{2}(\Omega)} =\langle A^{*}({t},x,{\n})w,v\rangle_{L^{2}(\Omega)}+\langle w, B({t},x,\PD_{\nu})v\rangle_{L^{2}(\PD\Omega)}
		\end{aligned}
	\end{align}
	where $B({t},x,\PD_{\nu})$ is the boundary operator we arrive at after integration by parts. {The first term on the right hand side of \eqref{identity for showing vanishing of w on boundary} is $0$ by \eqref{cokernel equation 1}. Next we show that $w=0$ on $\PD \Omega$. Let $u$ be an arbitrary vector field on $\PD \Omega$ with $C^{\infty}(\PD\Omega)$ coefficients. We show that there exists a vector field $v$ in $\Omega$ with $C^{\infty}$ coefficients such that} 
	\Beq\label{Boundary ODE}
			B({t},x,\PD_{\nu})v=u, \ \mbox{on}\ \PD\Omega,\ \mbox{and}\  v|_{\PD\Omega}=0. 
		\Eeq
	The boundary operator $B$ has a smooth extension to a small enough neighbourhood of the boundary. With this extension, we can consider \eqref{Boundary ODE} as an initial value problem for a system of first order ODEs with smooth coefficients, the solution of which exists in a small enough neighborhood of the boundary. This solution can now be extended smoothly to all of $\Omega$ which we denote by $v$.
	Using this in \eqref{identity for showing vanishing of w on boundary}, we get that $w|_{\PD\Omega}=0$.
	Thus, finally to show that the co-kernel of $A({t},x;\PD)$ is $0$,  we have to show that the following BVP 
	\begin{align}\label{Equation for A*}
		\begin{aligned}
			\begin{cases}
				A^{*}({t},x;\PD)w=0 \mbox{ for }  ({t},x)\in\Omega\\
				w({t},x)=0 \mbox{ for } {(t,x)}\in\PD\Omega
			\end{cases}
		\end{aligned}
	\end{align}
	has only the zero solution where $A^{*}({t},x;\PD)$ is the adjoint for operator $A({t},x;\PD)$.  Using the expression for $A^{*}({t},x;\PD)$ from \eqref{Adjoint to A} in \eqref{Equation for A*}, we have the following set of equations for $w_{j}$ for $0\leq j\leq n$ with zero Dirichlet boundary condition.
	\begin{align}\label{Equation for w}
		\begin{aligned}
			&\Delta w_{0}+(\alpha-\beta)\PD_{0}^{2}w_{0}+\beta\PD_{k0}^{2}w_{k}=0\\
			&\Delta w_{j}+(\alpha-\beta)\PD_{0j}w_{0}+\beta\PD_{kj}^{2}w_{k}=0,\ 1\leq j\leq n.
		\end{aligned}
	\end{align}
	Now multiplying the first equation in \eqref{Equation for w} by $w_{0}$ and second equation by $w_{j}$ and integrating over $\Omega$, we have 
	\begin{align*}
		\begin{aligned}
			&\int\limits_{\Omega}\lvert\nabla w_{0}({t},x)\rvert^{2}\D {t}\D x+(\alpha-\beta)\int\limits\PD_{0}w_{0}({t},x)\PD_{0}w_{0}({t},x)\D {t}\D x+\beta\sum_{k=0}^{n}\int\limits_{\Omega}\PD_{0}w_{0}({t},x)\PD_{k}w_{k}({t},x)\D {t}\D x=0;\\
			&\int\limits_{\Omega}\lvert\nabla w_{j}({t},x)\rvert^{2}\D {t}\D x+(\alpha-\beta)\int\limits_{\Omega}\PD_{0}w_{0}({t},x)\PD_{j}w_{j}({t},x)\D {t}\D x+\beta\sum_{k=0}^{n}\PD_{j}w_{j}({t},x)\PD_{k}w_{k}({t},x)\D {t}\D x;\ 1\leq j\leq n.
		\end{aligned}
	\end{align*}
	Adding the above set of equations and substituting the expressions for $\alpha$ and $\beta$, we get 
	\begin{align*}
		\begin{aligned}
			\sum_{j=0}^{n}\int\limits_{\Omega}\lvert\nabla w_{j}({t},x)\rvert^{2}\D {t}\D x+\frac{4}{n-1}\int\limits_{\Omega}\nabla\cdot w({t},x)\PD_{0}w_{0}({t},x)\D {t}\D x+\lb 1-\frac{2}{n-1}\rb\int\limits_{\Omega}\lvert\nabla\cdot w({t},x)\rvert^{2}\D {t}\D x=0.
		\end{aligned}
	\end{align*}
	This equation is exactly the same as that of \eqref{Quadratic in div v}.  Hence repeating the same arguments as before, we conclude that 
	$w({t},x)=0$. Thus we have co-kernel$(A)=\{0\}$ for $n\geq 4$. This completes the proof of the decomposition theorem for $n\geq 4$. 
\end{proof}
\begin{proof}[Proof of Theorem \ref{Main Theorem 2 for LRT of 2-tensor fields in R3}]
	Now combining the results of Theorems \ref{Main Theorem for LRT of 2-tensor fields in R3} and \ref{DT}, we conclude Theorem \ref{Main Theorem 2 for LRT of 2-tensor fields in R3}. For, given $F\in C^{\infty}(\overline{\Omega},S^{2}\Rb^{1+n})$, by Theorem \ref{Main Theorem 2 for LRT of 2-tensor fields in R3}, we can decompose $F=\wt{F}+\lambda g+ \D v$, with $\wt{F}, \lambda, v\in C^{\infty}(\overline{\Omega})$ satisfying $\delta(\wt{F})=\mbox{trace}(\wt{F})=0$ and $v|_{\PD \Omega}=0$, and $g$ is the Minkowski metric. It is straightforward to see that $\lambda g$ and $\D v$ above are in the kernel of the light ray transform; see \cite{LOSU} as well. The fact that $\D v$ with $v|_{\PD \Omega}=0$ lies in the kernel of the light ray transform follows by fundamental theorem of calculus and $\lambda g$ lies in the kernel because $g$ has signature $(-1,1,\cdots,1)$, and light ray transform integrates $F$ along lines in the direction $\wt{\theta}=(1,\theta)$ with $|\theta|=1$. Therefore, we conclude that $LF(t,x,\wt{\theta})=L\wt{F}(t,x,\wt{\theta})=0.$  Finally, to conclude, we apply Theorem \ref{Main Theorem for LRT of 2-tensor fields in R3} for $\wt{F}$, after extending $\wt{F}=0$ outside $\overline{\Omega}$.
	
\end{proof}

\section*{Acknowledgements}
\noindent Krishnan is supported in part by US NSF grant DMS 1616564 and India SERB Matrics grant MTR/2017/000837. The work of Vashisth is supported by NSAF grant  U1930402. The authors thank Plamen Stefanov for his comments and encouragement, and Vladimir Sharafutdinov for his interest in the problem. Finally, we thank the anonymous referees for the very useful comments that helped us immensely in improving the paper.

\end{document}